\documentclass[11pt]{article}
\usepackage[a4paper,margin=30mm]{geometry}
\usepackage{amsmath,amssymb,amsthm,mathtools}
\usepackage{booktabs,array}
\usepackage{graphicx}
\usepackage{enumitem}
\usepackage{hyperref}
\usepackage[nameinlink,capitalize,noabbrev]{cleveref}

\newtheorem{theorem}{Theorem}[section]
\newtheorem{proposition}[theorem]{Proposition}
\newtheorem{lemma}[theorem]{Lemma}
\newtheorem{corollary}[theorem]{Corollary}
\theoremstyle{definition}
\newtheorem{definition}[theorem]{Definition}
\newtheorem{example}[theorem]{Example}
\newtheorem{question}[theorem]{Question}
\theoremstyle{remark}
\newtheorem{remark}[theorem]{Remark}

\newcommand{\PP}{\mathbb P}
\newcommand{\RR}{\mathbb R}
\newcommand{\CC}{\mathbb C}
\newcommand{\cA}{\mathcal A}
\newcommand{\cB}{\mathcal B}
\newcommand{\cN}{\mathcal N}

\newcommand{\card}[1]{\lvert #1\rvert}

\title{Simplicial arrangements in real projective three-space revisited}
\author{Marek Janasz and Piotr Pokora}
\date{\today}

\begin{document}
\maketitle

\begin{abstract}
In this paper we study irreducible simplicial arrangements of projective planes in
$\PP^3(\RR)$ from combinatorial and projective geometry viewpoints.  We
first formulate a simpliciality criterion in terms of incidences between
rank-two and rank-three flats, together with equivalent formulations using
face numbers, reduced restrictions, and characteristic polynomials.  We
also relate the classical planar restriction data of Gr\"unbaum--Shephard
to Ziegler multirestrictions.  Our principal result concerns the rank-four
special-vertex property: among the irreducible crystallographic Coxeter
arrangements of rank four, the arrangements of types $A_4$ and $B_4$ admit
a special vertex, whereas those of types $D_4$ and $F_4$ do not.  A
simplicial deletion chain inside $B_4$ supplies further irreducible
examples with a special vertex.  Finally, we compare rank-flat, Purdy-type,
and Gr\"unbaum--Shephard defects for these arrangements.
\end{abstract}



\section{Introduction}
\label{sec:introduction}

Simplicial hyperplane arrangements form a natural meeting point of
combinatorics, discrete geometry, and the theory of reflection groups.  A
real projective arrangement is called simplicial if every chamber of its
complement is a simplex.  Reflection arrangements provide the most
prominent examples, but already in projective dimension three the
combinatorics of simpliciality becomes considerably richer than in the case
of line arrangements in the projective plane.  For an arrangement of
planes in $\PP^3(\RR)$ it is not enough to record the multiplicities of
intersection points and lines separately: the incidences between rank-two
and rank-three flats enter essentially.

The classical theory of simplicial arrangements in projective three-space
was developed in particular by Gr\"unbaum--Shephard
\cite{GrunbaumShephard}, and also in the PhD thesis of Hunt \cite{Hunt}, but from an algebraic geometry perspective.
Their work contains a large
amount of geometric and combinatorial information, including catalogue
data, planar sections, and numerical relations between the various strata
of an arrangement.  At the same time, the modern theory of hyperplane
arrangements provides a different set of tools, involving intersection
lattices, characteristic polynomials, freeness, and restrictions of
multiarrangements; see, for instance, \cite{DimcaBook,OrlikTerao}.  One
purpose of the present paper is to place the classical rank-four incidence
data into this modern framework and to make explicit the relations that will be
used in our study of special vertices.

Let $\cA$ be an essential arrangement of $n$ projective planes in
$\PP^3(\RR)$.  We denote by $t_p$ the number of vertices contained in
exactly $p$ planes and by $t_{pq}$ the number of incidences $(x,L)$ for
which $x$ is a $p$-fold vertex and $L$ is a $q$-fold intersection line
through $x$.  An important consequence of the face-counting relations is
the following characterization:
\[
 n+\sum_{p\geq 3}(2-p)t_p+
 \sum_{q\geq 3}(q-2)\sum_{p\geq 3}t_{pq}=0
\]
if and only if $\cA$ is simplicial -- see
Theorem~\ref{thm:simpliciality-RP3}.  This is a rank-four analogue of the
familiar numerical criterion for simplicial line arrangements in
$\PP^2(\RR)$ proved by Melchior \cite{Melchior}.  We also formulate the same condition using reduced
restrictions and characteristic polynomials.  Moreover, the planar
restrictions displayed in the classical Gr\"unbaum--Shephard catalogue
admit a natural interpretation as the supports of Ziegler
multirestrictions: their labels retain multiplicity information that is
lost after passing to the reduced restriction.

Our main combinatorial motivation comes from the notion of an arrangement with a special
vertex, recently introduced by Cuntz \cite{Cuntz}.  In the
finite rank-four setting used here, a special-vertex structure is encoded
by a pair $(P,H_\infty)$ with $H_\infty\in\cA$ and $P\notin H_\infty$
such that, for every $H\in\cA\setminus\{H_\infty\}$, there exists a
hyperplane $H_P\in\cA$ containing $P$ with
\[
H\cap H_\infty=H_P\cap H_\infty.
\]
Thus the condition imposes a strong compatibility between the projective
incidence structure and an affine parallelism obtained by deleting
$H_\infty$.

The principal result determines this property for the irreducible
crystallographic Coxeter arrangements of rank four.

\begin{theorem}\label{thm:main-special-vertex}
Among the irreducible crystallographic Coxeter arrangements of rank four,
the arrangements of types $A_4$ and $B_4$ admit a special vertex, whereas
the arrangements of types $D_4$ and $F_4$ do not.  Moreover, the simplicial
deletion chain
\[
 \cA^3_1(13)\subset
 \cA^3_1(14)\subset
 \cA^3_1(15)\subset
 \cA^3_1(16)=\cA(B_4)
\]
consists entirely of irreducible simplicial arrangements admitting a
special vertex.
\end{theorem}

The positive part of Theorem~\ref{thm:main-special-vertex} follows from
explicit projective normal forms.  For the negative part we use a dual
formulation, namely if
\[
 \cN(\cA)=\{[a_H]\mid H\in\cA\}\subset\PP(V^*)
\]
is the projective configuration of normal covectors, then the existence of
a special-vertex pair is equivalent to the existence of a distinguished
point $q_\infty\in\cN(\cA)$ and a plane $\Pi\subset\PP(V^*)$, with
$q_\infty\notin\Pi$, satisfying a simple collinearity condition with all
remaining points of $\cN(\cA)$.  This converts the non-existence problem
into a finite exact incidence computation.  Applied to the projectivized
root configurations of types $D_4$ and $F_4$, it excludes every possible
special-vertex pair.

A second theme concerns numerical invariants of the intersection lattice.
If $\ell(\cA)$ and $p(\cA)$ denote respectively the numbers of rank-two
and rank-three projective flats, we consider the rank-flat difference
\[
 G(\cA)=p(\cA)-\ell(\cA)
\]
and the Purdy defect
\[
 \Delta(\cA)=p(\cA)-\ell(\cA)+\card{\cA}+2.
\]
The latter is motivated by the Purdy-type incidence problems studied in
\cite{Purdy1981, Purdy1986}.  The corresponding inequality fails for complex projective
hyperplane arrangements \cite{MMPP}; it is therefore natural to ask
whether the additional real, irreducible, and simplicial hypotheses force
non-negativity.  For the arrangements studied here all Purdy defects are
non-negative.  The deletion chain inside $B_4$ is particularly
transparent: for $13\leq n\leq16$ one has
\[
 \Delta\bigl(\cA^3_1(n)\bigr)=16-n,
\]
so that the full $B_4$ arrangement is the equality case within this
family.  We finally compare these data with the Gr\"unbaum--Shephard
defect, which instead measures the proportion of ordinary intersection
lines.

The paper is organized as follows.  Section \ref{sec:preliminaries}
fixes notation and recalls the basic invariants.  Section
\ref{sec:simplicial-RP3} gives the simpliciality criteria.  Section
\ref{sec:ziegler-restrictions} discusses Ziegler multirestrictions, while
Section \ref{sec:special} introduces the special-vertex condition.
Sections \ref{sec:positive} and \ref{sec:negative} prove
Theorem~\ref{thm:main-special-vertex}.  Section \ref{sec:defects}
studies the three numerical defects and concludes with open questions.

\section{Simplicial arrangements in projective three-space}
\label{sec:preliminaries}
\subsection{Basics}

Let $V$ be a real vector space of dimension four.  A projective hyperplane
arrangement in $\PP(V)\cong\PP^3(\RR)$ is a finite set
\[
 \cA=\{H_1,\dots,H_n\}
\]
of distinct projective hyperplanes.  Choosing a nonzero covector
$a_H\in V^*$ with $H=\PP(\ker a_H)$ for each $H\in\cA$, one obtains the
defining polynomial
\[
 Q_{\cA}=\prod_{H\in\cA}a_H,
\]
well-defined up to a nonzero scalar.

The arrangement is \emph{essential} if the covectors $a_H$ span $V^*$.
Equivalently, the corresponding central arrangement in $V$ has rank four.
It is \emph{reducible} if, after a linear change of coordinates, its central
cone decomposes as a product of two nonempty lower-rank arrangements;
otherwise it is \emph{irreducible}.

The complement
\[
 \PP^3(\RR)\setminus\bigcup_{H\in\cA}H
\]
is a disjoint union of open chambers and their closures are convex projective
polyhedra after passing to a suitable affine chart.

\begin{definition}
An essential arrangement $\cA$ in $\PP^3(\RR)$ is \emph{simplicial} if the
closure of every chamber is combinatorially a tetrahedron.  Equivalently,
every chamber has exactly four facets.
\end{definition}

Simpliciality can be expressed combinatorially in terms of the face data of
the induced cell decomposition; see \cite{CuntzGeis}.  For the present
paper, we use only the geometric definition and the known simpliciality of
the finite real reflection arrangements.

We denote by $L_2(\cA)$ the set of rank-two projective flats
(intersection lines) and by $L_3(\cA)$ the set of rank-three projective
flats (intersection points, or vertices).  We write
\[
 \ell(\cA)=\card{L_2(\cA)},
 \qquad
 p(\cA)=\card{L_3(\cA)}.
\]
The multiplicity and incidence refinements of these numbers are introduced
in Section~\ref{sec:simplicial-RP3}, and the corresponding numerical
defects are considered in Section~\ref{sec:defects}.

\subsection{Freeness and characteristic polynomials}
\label{subsec:homological}

Although the arrangements considered in this paper are projective, the
standard homological invariants are attached to their central cones.  Let
$\widehat{\cA}$ be the central arrangement in $V\cong\RR^4$ corresponding
to $\cA$, and we put
\[
 S=\RR[t,x,y,z].
\]
After choosing defining linear forms $a_H\in V^*$, its defining polynomial
is the same homogeneous polynomial
\[
 Q_{\cA}=\prod_{H\in\cA}a_H,
 \qquad \deg Q_{\cA}=n=\card{\cA}.
\]
We grade derivations by the polynomial degree of their coefficient forms.
The module of logarithmic derivations is
\[
 D(\widehat{\cA})
 =
 \bigl\{\theta\in\operatorname{Der}_{\RR}(S)
 \mid \theta(Q_{\cA})\in Q_{\cA} S\}.
\]
We call the projective arrangement $\cA$ \emph{free} if
$D(\widehat{\cA})$ is a free graded $S$-module.  If
\[
 D(\widehat{\cA})
 \cong
 \bigoplus_{i=1}^{4} S(-d_i),
\]
then the integers $(d_1,d_2,d_3,d_4)$ are the \emph{exponents} of
$\cA$, denoted $\exp(\cA)$.  Since $\widehat{\cA}$ is central, the Euler
derivation
\[
 \theta_E
 =
 t\frac{\partial}{\partial t}
 +x\frac{\partial}{\partial x}
 +y\frac{\partial}{\partial y}
 +z\frac{\partial}{\partial z}
\]
belongs to $D(\widehat{\cA})$.  For an essential free rank-four
arrangement one exponent is therefore equal to $1$, and we write
\[
 \exp(\cA)=(1,d_2,d_3,d_4).
\]
Saito's criterion \cite[Theorem~1.8(ii)]{KS1} gives, in particular,
\[
 1+d_2+d_3+d_4=n.
\]
We shall also use the characteristic polynomial.  Let
$L(\widehat{\cA})$ be the intersection lattice of the central arrangement,
ordered by reverse inclusion, and let $\mu$ be its M\"obius function.  Then
\[
 \chi_{\widehat{\cA}}(u)
 =
 \sum_{X\in L(\widehat{\cA})}\mu(X)u^{\dim X}.
\]
It depends only on the intersection lattice.  Because $\widehat{\cA}$ is a
nonempty central arrangement, $u-1$ divides
$\chi_{\widehat{\cA}}(u)$, and we write
\[
 \overline\chi_{\cA}(u)
 =
 \frac{\chi_{\widehat{\cA}}(u)}{u-1}
\]
for the reduced characteristic polynomial.

The fundamental link with freeness is Terao's Factorization Theorem \cite{terao}:
if $\cA$ is free with exponents $(d_1,d_2,d_3,d_4)$, then
\[
 \chi_{\widehat{\cA}}(u)
 =
 \prod_{i=1}^{4}(u-d_i).
\]
Hence, in the essential rank-four case,
\[
 \chi_{\widehat{\cA}}(u)
 =(u-1)(u-d_2)(u-d_3)(u-d_4),
 \qquad
 \overline\chi_{\cA}(u)
 =(u-d_2)(u-d_3)(u-d_4).
\]
In particular, freeness determines a complete linear factorization of the
characteristic polynomial over the rationals, and the exponents determine
its roots.  The converse is false in general: integral factorization of
$\chi_{\widehat{\cA}}(u)$ is a necessary condition for freeness, but by
itself it does not imply that $D(\widehat{\cA})$ is free.  This distinction
will be important whenever characteristic-polynomial data are used as a
combinatorial test rather than as a freeness criterion. For concise accounts of these facts, we refer to two fundamental monographs \cite{DimcaBook,OrlikTerao}.
\subsection{Coxeter arrangements in the classical catalogue}
\label{subsec:coxeter-catalogue}

From a modern point of view, five entries in the Gr\"unbaum--Shephard catalogue are precisely the
projectivizations of the irreducible finite real reflection arrangements of
rank four.  In the notation used in the catalogue one has
\[
\cA^3_1(10)=\cA(A_4),\qquad
\cA^3_1(12)=\cA(D_4),\qquad
\cA^3_1(16)=\cA(B_4),
\]
\[
\cA^3_1(24)=\cA(F_4),\qquad
\cA^3_1(60)=\cA(H_4).
\]
The noncrystallographic type $H_4$ is included here for context; the main
special-vertex theorem concerns the four crystallographic rank-four types.
Their basic numerical data are summarized in
Table~\ref{tab:coxeter-catalogue}.
\begin{table}[htbp]
\centering
\renewcommand{\arraystretch}{1.15}
\begin{tabular}{c|c|r|r|c}
\toprule
catalogue notation & Coxeter type & $n$  & $2f_3$ & exponents\\
\midrule
$\cA^3_1(10)$ & $A_4$ & 10 &  120 & $(1,2,3,4)$\\
$\cA^3_1(12)$ & $D_4$ & 12 & 192 & $(1,3,3,5)$\\
$\cA^3_1(16)$ & $B_4$ & 16 &  384 & $(1,3,5,7)$\\
$\cA^3_1(24)$ & $F_4$ & 24 &  1152 & $(1,5,7,11)$\\
$\cA^3_1(60)$ & $H_4$ & 60 & 14400 & $(1,11,19,29)$\\
\bottomrule
\end{tabular}
\caption{The irreducible rank-four Coxeter arrangements occurring in the
Gr\"unbaum--Shephard catalogue. Here $f_3$ denotes the number of projective
chambers.}
\label{tab:coxeter-catalogue}
\end{table}

Indeed, if $W$ is the corresponding finite real reflection group, then its
central reflection arrangement in $\RR^4$ has exactly $|W|$ chambers.
Projectivization identifies every chamber with its antipodal chamber and hence
\[
2f_3=|W|.
\]
The values in the fourth column of Table~\ref{tab:coxeter-catalogue} are
therefore
\[
120,\ 192,\ 384,\ 1152,\ 14400,
\]
which are the orders of the Coxeter groups of types
$A_4,D_4,B_4,F_4,H_4$, respectively.

All five arrangements are free; their exponents are the Coxeter
exponents displayed in the last column of
Table~\ref{tab:coxeter-catalogue}; see, for example, \cite{OrlikTerao}. By Terao's factorization theorem, the last column also determines their characteristic polynomials.  For later use, we record those of the $D_4$ and
$F_4$ arrangements:
\[
 \chi_{\widehat{\cA(D_4)}}(u)=(u-1)(u-3)^2(u-5),
 \qquad
 \chi_{\widehat{\cA(F_4)}}(u)=(u-1)(u-5)(u-7)(u-11).
\]

\section{A simpliciality criterion for arrangements in \texorpdfstring{$\PP^3(\RR)$}{P3(R)}}
\label{sec:simplicial-RP3}
The situation in $\PP^3(\RR)$ is slightly more involved compared with the case of simplicial line arrangements in the real projective plane. In the
two-dimensional case, simpliciality can be expressed only in terms of
the numbers of intersection points of the various multiplicities. In
dimension three, one must also take into account the incidence relation
between intersection lines and intersection points, and this will be visible soon. 

Let
\[
\cA=\{H_1,\ldots,H_n\}
\]
be an essential arrangement of real projective planes in $\PP^3(\RR)$.
We denote by
\[
(f_0,f_1,f_2,f_3)
\]
the $f$-vector of the cell decomposition induced by $\cA$. Thus $f_i$ is
the number of $i$-dimensional cells and $f_3$ is the number of chambers.
We retain the definition of simpliciality given in the preceding section. We start with the following, probably folkloric, numerical characterization of the simpliciality.

\begin{proposition}
\label{prop:face-simpliciality-RP3}
Let $\cA$ be an essential arrangement of projective planes in
$\PP^3(\RR)$. Then the following conditions are equivalent:
\begin{enumerate}
    \item $\cA$ is simplicial;
    \item $f_2=2f_3$;
    \item $2f_0-2f_1+f_2=0$.
\end{enumerate}
\end{proposition}

\begin{proof}
For a chamber $C$, let $w(C)$ denote the number of its two-dimensional
walls. Since $\cA$ is essential, every chamber has at least four
walls. Moreover, every two-dimensional face is incident with exactly
two chambers. Consequently,
\[
2f_2=\sum_C w(C)\geq 4f_3.
\]
Equality holds if and only if every chamber has exactly four walls,
which is precisely the simpliciality condition for $\cA$. This proves the
equivalence of the first two conditions.

Next, since
\[
\chi(\PP^3(\RR))=0,
\]
The classical algebraic topology Euler's relation for the induced cell decomposition reads as
\[
f_0-f_1+f_2-f_3=0.
\]
If $f_2=2f_3$, then
\[
f_0-f_1+\frac{f_2}{2}=0,
\]
which is equivalent to
\[
2f_0-2f_1+f_2=0.
\]
The converse follows in the same way.
\end{proof}
For every $H\in\cA$, we consider the reduced restriction
\[
\cA^H
=
\{H\cap H'\mid H'\in\cA\setminus\{H\}\}_{\mathrm{red}}.
\]
It is an arrangement of projective lines in
\[
H\simeq \PP^2(\RR).
\]
We write $r(\cB)$ for the number of chambers of a real
projective arrangement $\cB$.

\begin{proposition}
\label{prop:restriction-criterion}
For an essential arrangement $\cA$ of projective planes in
$\PP^3(\RR)$, one has
\[
f_2(\cA)
=
\sum_{H\in\cA}r(\cA^H).
\]
Consequently,
\[
\cA \text{ is simplicial}
\quad\Longleftrightarrow\quad
\sum_{H\in\cA}r(\cA^H)
=
2r(\cA).
\]
\end{proposition}

\begin{proof}
Every two-dimensional face of $\cA$ is contained in a unique
plane $H\in\cA$. The two-dimensional faces contained in $H$ are
precisely the chambers of the induced line arrangement $\cA^H$.
Therefore,
\[
f_2(\cA)
=
\sum_{H\in\cA}r(\cA^H).
\]
Since $f_3(\cA)=r(\cA)$, the second assertion follows
from Proposition~\ref{prop:face-simpliciality-RP3}.
\end{proof}

This formulation admits a characteristic-polynomial interpretation.
Let $\widehat{\cA}$ be the central arrangement in $\RR^4$
associated with $\cA$. Similarly, let
$\widehat{\cA^H}$ denote the central arrangement in $\RR^3$
associated with the projective line arrangement $\cA^H$.
Zaslavsky's Theorem \cite[Theorem A]{zaslavsky} gives
\[
r(\cA)
=
\frac{1}{2}\chi_{\widehat{\cA}}(-1)
\]
and
\[
r(\cA^H)
=
-\frac{1}{2}\chi_{\widehat{\cA^H}}(-1).
\]
We therefore obtain the following consequence.

\begin{corollary}[{\cite[Corollary 2.4]{CuntzGeis}}]
\label{cor:characteristic-polynomial-criterion}
The arrangement $\cA$ is simplicial if and only if
\[
-\sum_{H\in\cA}
\chi_{\widehat{\cA^H}}(-1)
=
2\chi_{\widehat{\cA}}(-1).
\]
In particular, simpliciality is determined by the intersection lattice
of $\cA$.
\end{corollary}

\begin{remark}
No freeness hypothesis is involved in Corollary~\ref{cor:characteristic-polynomial-criterion}:
the characteristic polynomial is being used here through Zaslavsky's region
formula, not through Terao's factorization theorem.  Thus this combinatorial
simpliciality criterion applies equally to free and non-free arrangements, and this is crucial.
\end{remark}

We now refine the rank-two and rank-three flat counts by multiplicity and incidence. For
$L\in L_2(\cA)$ and $x\in L_3(\cA)$, put
\[
m(L)=\bigl|\{H\in\cA\mid L\subset H\}\bigr|,
\qquad
\nu(x)=\bigl|\{H\in\cA\mid x\in H\}\bigr|.
\]

Thus $m(L)$ is the multiplicity of an intersection line and $\nu(x)$ is the
number of planes passing through $x$.

Following the notation of Hunt \cite[Section 2.1]{Hunt}, for $q\geq 2$ and
$p\geq 3$ we set
\[
t_q^{(1)}
=
\bigl|\{L\in L_2(\cA)\mid m(L)=q\}\bigr|,
\]
\[
t_p
=
\bigl|\{x\in L_3(\cA)\mid \nu(x)=p\}\bigr|,
\]
and
\[
t_{pq}
=
\bigl|\{(x,L)\in L_3(\cA)\times L_2(\cA)
\mid x\in L,\ \nu(x)=p,\ m(L)=q\}\bigr|.
\]
Thus $t_{pq}$ is an incidence number.%: if a $p$-fold point lies on two distinct $q$-fold lines, both incidences are counted.

These numerical data satisfy the two elementary identities recorded in
\cite[Section 2.1]{Hunt}. Counting pairs of planes gives
\begin{equation}
\label{eq:GS-pair-count}
\sum_{q\geq 2} t_q^{(1)}\binom{q}{2}
=
\binom{n}{2}.
\end{equation}
Counting triples of planes gives
\begin{equation}
\label{eq:GS-triple-count}
\sum_{p\geq 3} t_p\binom{p}{3}
-
\sum_{q\geq 3}
\left(
\sum_{p\geq 3}t_{pq}-t_q^{(1)}
\right)
\binom{q}{3}
=
\binom{n}{3}.
\end{equation}
Indeed, the first sum in \eqref{eq:GS-triple-count} counts a triple of
planes once at every vertex lying on their common intersection. If the
three planes belong to a $q$-fold line, then this triple has been counted
once for every vertex on that line, whereas it should contribute only
once. For a fixed $q$, the excess is therefore
\[
\left(
\sum_{p\geq 3}t_{pq}-t_q^{(1)}
\right)\binom{q}{3},
\]
which explains the correction term. As a consistency check, for an arrangement in general position we have
\[
t_2^{(1)}=\binom{n}{2},\qquad t_3=\binom{n}{3},
\]
and all higher multiplicity data vanish.
\begin{example} Let us illustrate the meaning of the incidence numbers $t_{pq}$ on the Coxeter arrangement $\cA^3_1(16)$. 
Recall that $t_{63}$ counts pairs $(x,L)$ such that $x$ is a $6$-fold vertex, $L$ is a $3$-fold intersection line, and $x\in L$. The $6$-fold vertices of $\cA^3_1(16)$ are represented by 
\[ [\varepsilon_1:\varepsilon_2:\varepsilon_3:\varepsilon_4], \qquad \varepsilon_i\in\{\pm1\}, \] 
where simultaneous multiplication of all coordinates by $-1$ gives the same projective point. 
Hence there are $2^4/2=8$ such vertices. Consider, for instance, \[ P=[1:1:1:1]. \] 
Exactly six hyperplanes of $\cA^3_1(16)$ pass through $P$, namely \[ t-x=0,\quad t-y=0,\quad t-z=0,\quad x-y=0,\quad x-z=0,\quad y-z=0. \] Among the intersection lines through $P$, precisely four have multiplicity three: \[ \{t=x=y\},\qquad \{t=x=z\},\qquad \{t=y=z\},\qquad \{x=y=z\}. \] For example, the line $\{t=x=y\}$ is contained in the three hyperplanes \[ t-x=0,\qquad t-y=0,\qquad x-y=0. \] Thus every $6$-fold vertex is incident with exactly four $3$-fold lines. Consequently, \[ t_{63}=8\cdot4=32. \] Equivalently, $\cA^3_1(16)$ has $16$ triple lines, each containing two $6$-fold vertices, which gives the same incidence count $t_{63}=16\cdot2=32$. \end{example}
\begin{lemma}
\label{lem:face-numbers-incidence}
With the notation above, one has
\[
f_0
=
\sum_{p\geq 3}t_p,
\]
\[
f_1
=
\sum_{p\geq 3}\sum_{q\geq 2}t_{pq},
\]
and
\[
f_2
=
n
+
\sum_{p\geq 3}\sum_{q\geq 2}q\,t_{pq}
-
\sum_{p\geq 3}p\,t_p.
\]
\end{lemma}

\begin{proof}
The first identity is immediate. We first note that every
$L\in L_2(\cA)$ contains a vertex. Indeed, in the central picture let
$X$ be the two-dimensional subspace corresponding to $L$. The
restrictions to $X$ of the defining forms of the hyperplanes not
containing $X$ must span $X^*$; otherwise they would have a nonzero
common kernel in $X$, and that vector would lie in every hyperplane of
the arrangement, contradicting essentiality. Hence at least two
distinct rank-three flats lie on $L$.

For the second identity, every one-dimensional cell lies on a unique
intersection line. A projective line containing $r$ vertices is divided
by them into exactly $r$ one-dimensional cells.
Consequently, $f_1$ is the total number of incidences between vertices
and intersection lines, which is
\[
\sum_{p\geq 3}\sum_{q\geq 2}t_{pq}.
\]
For $H\in\cA$, let $\cA^H$ be the reduced restriction and, for a vertex
$x\in H$, put
\[
\lambda_H(x)
=
\bigl|\{L\in L_2(\cA)\mid x\in L\subset H\}\bigr|.
\]
Again the Euler formula for the projective line arrangement $\cA^H$ gives us
\[
r(\cA^H)
=
1+
\sum_{\substack{x\in L_3(\cA)\\x\in H}}
\bigl(\lambda_H(x)-1\bigr).
\]
Summing over $H\in\cA$ and using the fact
\[
f_2=\sum_{H\in\cA}r(\cA^H),
\]
we obtain
\[
f_2
=
n+
\sum_{x\in L_3(\cA)}
\left(
\sum_{\substack{L\in L_2(\cA)\\x\in L}}m(L)-\nu(x)
\right).
\]
Finally,
\[
\sum_{x\in L_3(\cA)}\nu(x)
=
\sum_{p\geq 3}p\,t_p
\]
and
\[
\sum_{x\in L_3(\cA)}
\sum_{\substack{L\in L_2(\cA)\\x\in L}}m(L)
=
\sum_{p\geq 3}\sum_{q\geq 2}q\,t_{pq},
\]
which gives the asserted formula for $f_2$.
\end{proof}

Combining Proposition~\ref{prop:face-simpliciality-RP3} with
Lemma~\ref{lem:face-numbers-incidence} gives us the announced simpliciality criterion.

\begin{theorem}[Simpliciality criterion in $\PP^3(\RR)$]
\label{thm:simpliciality-RP3}
Let
\[
\cA=\{H_1,\ldots,H_n\}
\]
be an essential arrangement of real projective planes in
$\PP^3(\RR)$. Then $\cA$ is simplicial if and only if
\begin{equation}
\label{eq:simpliciality-GS-data}
n
+
\sum_{p\geq 3}(2-p)t_p
+
\sum_{q\geq 3}(q-2)
\sum_{p\geq 3}t_{pq}
=0.
\end{equation}
\end{theorem}

\begin{proof}
By Proposition~\ref{prop:face-simpliciality-RP3}, the arrangement is
simplicial if and only if
\[
2f_0-2f_1+f_2=0.
\]
Using Lemma~\ref{lem:face-numbers-incidence}, we obtain
\[
\begin{aligned}
0= 2f_0-2f_1+f_2
&=
2\sum_{p\geq 3}t_p
-
2\sum_{p\geq 3}\sum_{q\geq 2}t_{pq}
+n
+
\sum_{p\geq 3}\sum_{q\geq 2}q\,t_{pq}
-
\sum_{p\geq 3}p\,t_p\\
&=
n
+
\sum_{p\geq 3}(2-p)t_p
+
\sum_{p\geq 3}\sum_{q\geq 2}(q-2)t_{pq}.
\end{aligned}
\]
The contribution with $q=2$ is zero, so the last double sum may be
restricted to $q\geq 3$, and this proves \eqref{eq:simpliciality-GS-data}.
\end{proof}

\begin{example}
\label{ex:four-planes-general-position}
Consider four planes in general position in $\PP^3(\RR)$. This is the
projective tetrahedral arrangement. We have
\[
n=4,
\qquad
t_2^{(1)}=6,
\qquad
t_3=4,
\]
while $t_q^{(1)}=0$ for $q\geq 3$ and $t_p=0$ for $p\geq 4$. All incidences
between vertices and intersection lines have $q=2$, and therefore the
last term in \eqref{eq:simpliciality-GS-data} vanishes. Hence
\[
4+(2-3)\cdot 4=0,
\]
so Theorem~\ref{thm:simpliciality-RP3} recovers the simpliciality of the
tetrahedral arrangement.
\end{example}

\begin{remark}
Theorem \ref{thm:simpliciality-RP3} may be viewed as the three-dimensional analogue of Melchior's classical criterion for line arrangements in the real projective plane \cite{Melchior}. More precisely, let $\mathcal{L}\subset \mathbb{P}^{2}(\mathbb{R})$ be an arrangement of $d\geq 3$ lines which is not a pencil. Then $\mathcal{L}$ is simplicial if and only if
\[
n_{2}=3+\sum_{r\geq 4}(r-3)n_{r},
\]
where $n_r$ denotes the number of intersection points of multiplicity $r$ in $\mathcal{L}$.
\end{remark}

\section{Restrictions and Ziegler multirestrictions}
\label{sec:ziegler-restrictions}

The reduced restrictions introduced above are also the natural point of
contact with the classical Gr\"unbaum--Shephard diagrams. For a fixed plane
$H\in\cA$, the induced line arrangement in $H$ consists of the
projective lines
\[
H\cap H',\qquad H'\in\cA\setminus\{H\},
\]
after repetitions have been removed. Thus, from a modern
arrangement-theoretic viewpoint, the planar arrangements displayed in the
catalogue record reduced restrictions $\cA^H$.

There is a slightly finer object which retains the repetitions. Define a
multiplicity on the lines of $\cA^H$ by
\[
m^H(L)
=
\bigl|\{H'\in\cA\setminus\{H\}\mid H\cap H'=L\}\bigr|.
\]
Then
\[
(\cA^H,m^H)
\]
is the projectivized Ziegler multirestriction of the corresponding central
arrangement to $H$; see \cite{Ziegler}. Since $m(L)$ denotes the number of
planes of $\cA$ containing the rank-two flat $L$, one has simply
\[
m^H(L)=m(L)-1
\qquad\text{for every }L\subset H.
\]
Hence the support of the Ziegler multirestriction is exactly the reduced
restriction $\cA^H$, while the multiplicity remembers how many planes
of $\cA$ induce the same line in $H$.

The multirestriction is also compatible with the freeness.  Ziegler's
Restriction Theorem \cite{Ziegler} tells us that if $\cA$ is a free arrangement with exponents
\[
 \exp(\cA)=(1,d_2,d_3,d_4),
\]
then, for every $H\in\cA$, the central rank-three multiarrangement represented
by the projectivized Ziegler multirestriction $(\cA^H,m^H)$ is free, with
exponents
\[
 \exp(\cA^H,m^H)=(d_2,d_3,d_4).
\]
Thus the multirestriction retains homological information which is lost by
passing only to the reduced support $\cA^H$.  Notice that this is a one-way
implication as stated: freeness of the ambient arrangement forces freeness
of the Ziegler multirestriction, whereas the reduced restriction alone does
not provide a converse freeness criterion in rank four, see \cite[Corollary 1.35]{Yoshinaga}.

This observation is particularly useful for the Gr\"unbaum--Shephard
catalogue. The authors do not organize their examples by explicit defining
polynomials. Instead, they describe geometric constructions, symmetry
orbits of planes, and the induced planar arrangements in representative
planes. Consequently, the figures and their
plane labels can be read as restriction data. Whenever the labels determine
which ambient planes induce a given line in $H$, retaining these repetitions
recovers the multiplicity $m^H$ and hence the corresponding Ziegler
multirestriction. This means that the classical catalogue contains more information than the
$f$-vectors and valence tables alone suggest! After reconstructing explicit
realizations, one can use the displayed restrictions as consistency checks
for the rank-two flats and, at the same time, as input for computations of
multiarrangement exponents, freeness, and restriction-theoretic invariants.
We stress that this is our modern reinterpretation of the catalogue: Ziegler
multirestrictions are not part of the terminology and they are not considered in 
\cite{GrunbaumShephard}.

\begin{example}[The arrangement $\mathcal A^3_1(10)$]
Consider the Coxeter arrangement of type $A_4$,
\[
Q_{\mathcal A}
=txyz(x-t)(y-t)(z-t)(x-y)(x-z)(y-z).
\]
In the notation of Gr\"unbaum--Shephard this is the simplicial
arrangement $\mathcal A^3_1(10)$.  Their construction from a regular
tetrahedron distinguishes four facet planes, denoted by $\alpha$, and
six symmetry planes, denoted by $\beta$.

Let
\[
H=\{t=0\},
\]
which may be chosen as one of the $\alpha$-planes.  The reduced
restriction to $H\simeq\mathbb P^2$ is
\[
\mathcal A^H=
\{x,y,z,x-y,x-z,y-z\},
\]
and hence is the six-line simplicial arrangement
$\mathcal A^2_1(6)$ displayed in Figure~\ref{fig:A4-restriction}.

\begin{figure}[htbp]
    \centering
    \includegraphics[scale=0.25]{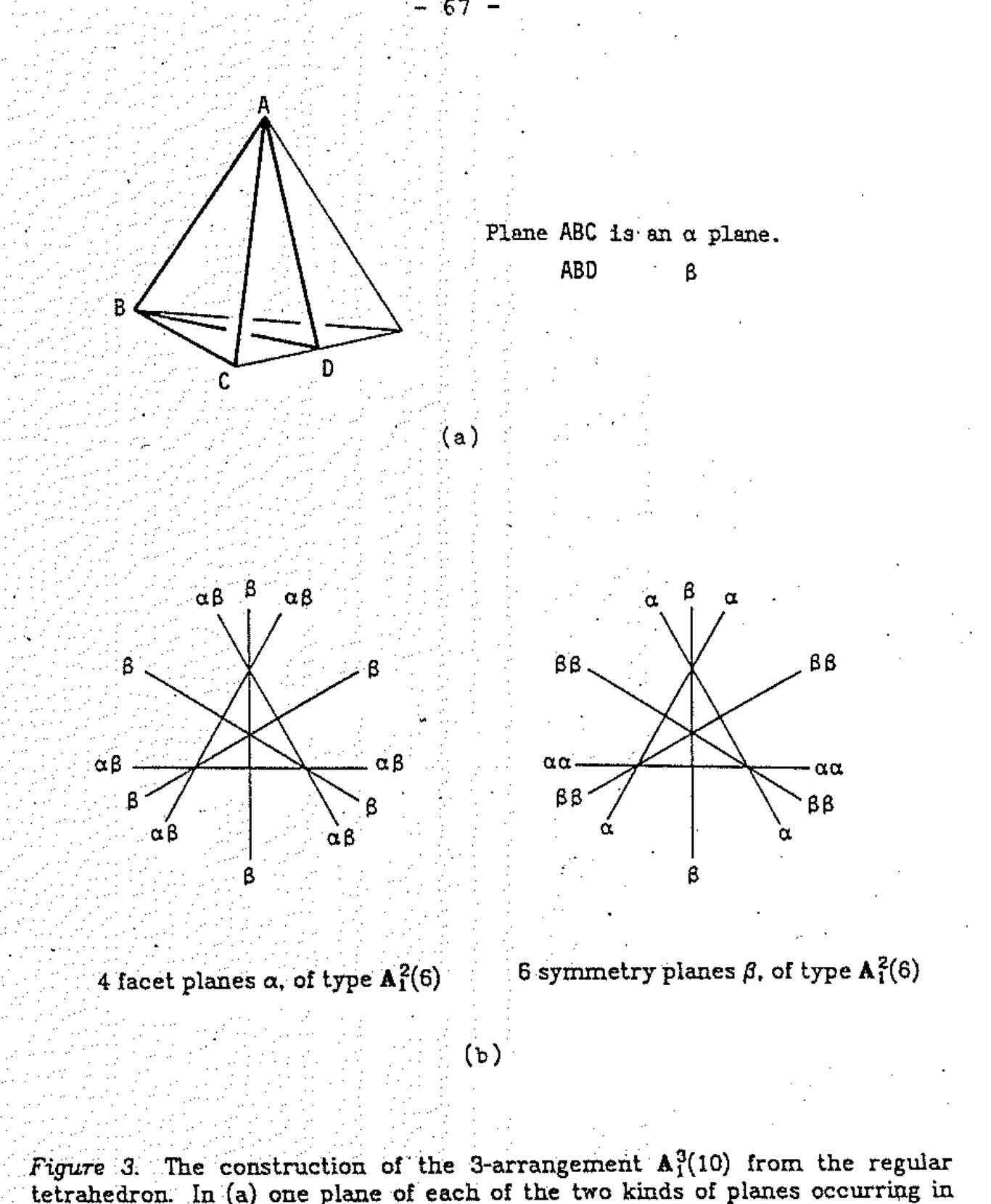}
    \caption{The arrangement $\mathcal A^3_1(10)=\mathcal A(A_4)$ and its planar restrictions, taken from \cite[Figure~3, p.~67]{GrunbaumShephard}.}
    \label{fig:A4-restriction}
\end{figure}
The labels in the original Gr\"unbaum--Shephard diagram contain more
information than the reduced restriction.  Indeed,
\[
H\cap\{x=0\}=H\cap\{x-t=0\},
\]
and similarly for $y$ and $z$.  Consequently
\[
m^H(x)=m^H(y)=m^H(z)=2,
\]
whereas
\[
m^H(x-y)=m^H(x-z)=m^H(y-z)=1.
\]
Thus the corresponding Ziegler multirestriction is represented by
\[
Q_{(\mathcal A^H,m^H)}
=x^2y^2z^2(x-y)(x-z)(y-z).
\]
In the classical diagram the three lines of multiplicity two are
precisely those marked $\alpha\beta$, while the remaining three lines
are marked $\beta$.  Thus the Gr\"unbaum--Shephard labels recover exactly
the multiplicity function $m^H$ which is forgotten by passing to the
reduced restriction. Finally, since
\[
\exp(\mathcal A)=(1,2,3,4),
\]
Ziegler's restriction theorem gives
\[
\exp(\mathcal A^H,m^H)=(2,3,4).
\]
\end{example}

\section{The rank-four special-vertex property}
\label{sec:special}
We now introduce the special-vertex condition in a form adapted to
projective arrangements.  This is the intrinsic projective version of the
coordinate model used by Cuntz in \cite{Cuntz}.  The crucial difference is the integrality convention for the affine shift sets.  Allowing proportional representatives, this causes no restriction: a hyperplane
$\alpha+s t=0$ with $s\neq0$ may be encoded using the proportional
representative $\alpha/s$ and the shift set $\{0,1\}$, while $s=0$
already gives the member of the trace class passing through $P$.  Thus a
trace class may be split into proportional representatives and rescaled
separately.  The intrinsic version is more convenient for the examples and
non-existence arguments below.

\begin{definition}\label{def:special}
Let $\cA$ be an essential arrangement in $\PP^3(\RR)$.  We say that a pair $(H_\infty,P)$, where $H_\infty\in\mathcal A$ and
$P\not \in H_\infty$ is a point, is a special-vertex
pair if for every
\[
H\in\mathcal A\setminus\{H_\infty\}
\]
there exists a hyperplane $H_P\in\mathcal A$ containing $P$ such that
\[
H\cap H_\infty=H_P\cap H_\infty.
\]
Equivalently, every trace on $H_\infty$ occurring in the arrangement
is represented by a hyperplane of $\mathcal A$ passing through $P$.

In this situation, we say that $\cA$ is a \emph{special-vertex arrangement}.
\end{definition}

The following lemma gives the corresponding coordinate normal form.  Here
$[t:x:y:z]$ are homogeneous coordinates on $\PP^3$, and
$(\RR^3)^*$ is identified with the linear forms in $x,y,z$.

\begin{lemma}\label{lem:intrinsic-special-vertex}
An essential arrangement $\cA$ in $\PP^3(\RR)$ admits the
special-vertex pair
\[
 P=[1:0:0:0],\qquad H_\infty=\{t=0\}
\]
if and only if it can be written in the form
\begin{equation}
 \cA
 =\{t=0\}\cup
 \{\alpha(x,y,z)+kt=0\mid \alpha\in\mathcal R,\ k\in\Sigma_\alpha\}.
\label{eq:special-normal-form}
\end{equation}
where $\mathcal R\subset(\RR^3)^*$ is a finite set of chosen,
pairwise non-proportional nonzero covectors spanning $(\RR^3)^*$, and
each $\Sigma_\alpha\subset\RR$ is finite, nonempty, and contains $0$.
After an arbitrary projective change of coordinates, the same statement
holds for any special-vertex pair $(P,H_\infty)$.
\end{lemma}

\begin{proof}
Assume first that $(P,H_\infty)$ is a special-vertex pair and choose
coordinates with $H_\infty=\{t=0\}$ and $P=[1:0:0:0]$.  Every hyperplane
through $P$ has an equation independent of $t$, say
$\alpha(x,y,z)=0$.  If another hyperplane $H$ has the same trace on
$H_\infty$, then the restriction of its defining form to $t=0$ is
proportional to $\alpha$.  After rescaling, $H$ has an equation
\[
 \alpha(x,y,z)+kt=0
\]
for some $k\in\RR$.  Choose one representative $\alpha$ for each trace
class.  Grouping the hyperplanes by their trace on $H_\infty$ produces
the finite sets $\Sigma_\alpha$, and the hyperplane through $P$ in each
class gives $0\in\Sigma_\alpha$.  Essentiality implies that the chosen
covectors $\alpha$ span $(\RR^3)^*$.

Conversely, suppose that $\cA$ has the form \eqref{eq:special-normal-form}.
For a hyperplane
$H=\{\alpha(x,y,z)+kt=0\}$ take
$H_P=\{\alpha(x,y,z)=0\}$, which belongs to $\cA$ because
$0\in\Sigma_\alpha$.  Then $P\in H_P$ and
$H\cap H_\infty=H_P\cap H_\infty$.  Hence $(P,H_\infty)$ is a
special-vertex pair.
\end{proof}

\begin{remark}
The pair $(P,H_\infty)$ is part of the structure.  An arrangement can
admit more than one such pair, and a hyperplane which passes through the
special vertex for one choice may fail to do so for another.
\end{remark}

\section{Irreducible simplicial arrangements with a special vertex}
\label{sec:positive}
In this section we study the notion of arrangements with a special vertex in the class of irreducible simplicial arrangements. We present here some instructive examples.
\subsection{The arrangement of type \texorpdfstring{$A_4$}{A4}}

\begin{example}\label{ex:A4}
Consider the reflection arrangement of type $A_4$ in coordinates
$[t:x:y:z]$ with defining polynomial
\[
 Q_{A_4}
 =txyz(x-t)(y-t)(z-t)(x-y)(x-z)(y-z).
\]
This is the simplicial arrangement
$\cA^3_1(10)=\cA(A_4)$ in the notation of the catalogue
\cite{Geis}.  Take
\[
 H_\infty=\{t=0\},
 \qquad
 P=[1:0:0:0].
\]
The six hyperplanes
\[
 x=0,\quad y=0,\quad z=0,\quad
 x-y=0,\quad x-z=0,\quad y-z=0
\]
pass through $P$.  The remaining three hyperplanes are
\[
 x-t=0,\qquad y-t=0,\qquad z-t=0.
\]
Consequently,
\[
 \Sigma_x=\Sigma_y=\Sigma_z=\{0,-1\},
\]
whereas
\[
 \Sigma_{x-y}=\Sigma_{x-z}=\Sigma_{y-z}=\{0\}.
\]
Thus $\cA(A_4)$ is an irreducible simplicial special-vertex arrangement.
\end{example}

\subsection{The arrangement of type \texorpdfstring{$B_4$}{B4}}

\begin{example}\label{ex:B4}
The Coxeter arrangement of type $B_4$ is defined by
\[
 Q_{B_4}=txyz
 \prod_{\varepsilon=\pm1}
 (t+\varepsilon x)(t+\varepsilon y)(t+\varepsilon z)
 \prod_{\varepsilon=\pm1}
 (x+\varepsilon y)(x+\varepsilon z)(y+\varepsilon z).
\]
It is the simplicial arrangement
$\cA^3_1(16)=\cA(B_4)$.  Again take
\[
 H_\infty=\{t=0\},
 \qquad
 P=[1:0:0:0].
\]
The hyperplanes through $P$ form the $B_3$ arrangement
\[
 x=0,\quad y=0,\quad z=0,\quad
 x\pm y=0,\quad x\pm z=0,\quad y\pm z=0.
\]
The noncentral hyperplanes are
\[
 t\pm x=0,\qquad t\pm y=0,\qquad t\pm z=0.
\]
Equivalently,
\[
 \Sigma_x=\Sigma_y=\Sigma_z=\{-1,0,1\},
\]
and all other direction classes have shift set $\{0\}$.  Therefore
$\cA(B_4)$ is special-vertex.  It is also irreducible and simplicial.
\end{example}

\subsubsection{A simplicial deletion chain inside \texorpdfstring{$B_4$}{B4}}

Delete coordinate hyperplanes among
$\{t=0,x=0,y=0,z=0\}$ from $\cA(B_4)$.  Up to the symmetry of $B_4$,
deleting one, two, or three such hyperplanes produces the catalogue entries
\[
 \cA^3_1(15),\qquad \cA^3_1(14),\qquad \cA^3_1(13),
\]
respectively.  These arrangements are irreducible and simplicial.  Each
continues to satisfy the intrinsic criterion, although the convenient
choice of $H_\infty$ changes after deletion.

\begin{proposition}\label{prop:B4-chain}
The arrangements
\[
 \cA^3_1(13),\qquad \cA^3_1(14),\qquad
 \cA^3_1(15),\qquad \cA(B_4)
\]
are irreducible simplicial special-vertex arrangements.  The following
pairs give certificates for the three proper subarrangements.
\[
\begin{array}{c|c|c|c}
\toprule
\text{Arrangement}&\text{deleted hyperplanes}&H_\infty&P\\
\midrule
\cA^3_1(15)&t=0&t+x=0&[1:0:0:0]\\
\cA^3_1(14)&t=0,\ x=0&x+y=0&[0:1:0:0]\\
\cA^3_1(13)&t=0,\ x=0,\ y=0&y+z=0&[0:0:1:0]\\
\bottomrule
\end{array}
\]
\end{proposition}

\begin{proof}
The irreducibility and simpliciality are recorded in the catalogue
\cite{Geis}.  It remains to verify the special-vertex property, and by Lemma
\ref{lem:intrinsic-special-vertex} it is enough to check the displayed
pairs.

We give the computation for $\cA^3_1(15)=\cA(B_4)\setminus\{t=0\}$.
Choose
\[
 H_\infty=\{t+x=0\}
\]
and introduce $s=t+x$.  Then $H_\infty=\{s=0\}$ and the special vertex is
$P=[1:0:0:0]$ in the coordinates $[s:x:y:z]$.  The hyperplanes through $P$
are
\[
 x=0,\quad y=0,\quad z=0,\quad
 x\pm y=0,\quad x\pm z=0,\quad y\pm z=0.
\]
The remaining hyperplanes are
\[
 2x-s=0,
\]
\[
 y-x+s=0,\qquad x+y-s=0,
\]
\[
 z-x+s=0,\qquad x+z-s=0.
\]
Each has on $s=0$ the same trace as one of the central hyperplanes listed
above.  Hence Lemma~\ref{lem:intrinsic-special-vertex} applies.

For $\cA^3_1(14)=\cA(B_4)\setminus\{t=0,x=0\}$ take
\[
 H_\infty=\{x+y=0\},\qquad P=[0:1:0:0].
\]
The hyperplanes through $P$ are
\[
 y=0,\ z=0,\ t\pm y=0,\ t\pm z=0,\ y\pm z=0.
\]
The remaining noncentral hyperplanes have the following partners through
$P$ with the same trace on $H_\infty$:
\[
\begin{array}{c|c@{\qquad}c|c}
H&H_P&H&H_P\\
\hline
t+x=0&t-y=0&t-x=0&t+y=0\\
x-y=0&y=0&x+z=0&y-z=0\\
x-z=0&y+z=0&&
\end{array}
\]
(the hyperplane $H_\infty$ itself is omitted from the table).  This
verifies the special-vertex condition.

Finally, for
$\cA^3_1(13)=\cA(B_4)\setminus\{t=0,x=0,y=0\}$ take
\[
 H_\infty=\{y+z=0\},\qquad P=[0:0:1:0].
\]
The hyperplanes through $P$ are
\[
 z=0,\ t\pm x=0,\ t\pm z=0,\ x\pm z=0.
\]
The remaining noncentral hyperplanes are paired as follows:
\[
\begin{array}{c|c@{\qquad}c|c}
H&H_P&H&H_P\\
\hline
t+y=0&t-z=0&t-y=0&t+z=0\\
x+y=0&x-z=0&x-y=0&x+z=0\\
y-z=0&z=0&&
\end{array}
\]
so Lemma~\ref{lem:intrinsic-special-vertex} applies again.  This proves the
claim for all three proper subarrangements.
\end{proof}
\begin{remark}
The examples above show, in particular, that the special-vertex condition is not a
reformulation of reducibility: it occurs for genuinely irreducible reflection arrangements and for
irreducible simplicial subarrangements of them.
\end{remark}
\section{Irreducible simplicial arrangements without a special vertex}
\label{sec:negative}
For the non-existence results it is convenient to pass to the dual configuration. Let
\[
 \cN(\cA)=\{[a_H]\mid H\in\cA\}\subset\PP(V^*)
\]
be the projective configuration of normal covectors.  For
$P\in\PP(V)$, let
\[
 \Pi_P=\{[a]\in\PP(V^*)\mid a(P)=0\}.
\]
Thus $\Pi_P$ is the projective plane dual to $P$.

\begin{lemma}
\label{lem:dual-special-vertex}
Let $\cA$ be an essential arrangement in $\PP(V)$ and put
$\cN=\cN(\cA)$.  Then $\cA$ is a special-vertex arrangement if and only if
there exist a point $q_\infty\in\cN$ and a projective plane
$\Pi\subset\PP(V^*)$, with $q_\infty\notin\Pi$, such that
\[
 \langle q_\infty,q\rangle\cap(\cN\cap\Pi)\neq\varnothing
 \qquad
 \text{for every }q\in\cN\setminus\{q_\infty\}.
\]
Moreover, if such a pair exists, then $\cN\cap\Pi$ spans $\Pi$.
\end{lemma}

\begin{proof}
Assume first that $(P,H_\infty)$ is a special-vertex pair.  Let
$q_\infty=[a_\infty]$ be dual to $H_\infty$ and put $\Pi=\Pi_P$.  A point
$[b]$ belongs to $\Pi$ exactly when the corresponding hyperplane contains
$P$.

Let $H=\ker(a)$ and $H_P=\ker(b)$ be members of the arrangement.  Their
traces on $H_\infty=\ker(a_\infty)$ agree if and only if the restrictions
of $a$ and $b$ to $\ker(a_\infty)$ are proportional.  This is equivalent
to linear dependence of $a_\infty,a,b$, hence to collinearity of
$q_\infty,[a],[b]$ in $\PP(V^*)$.  The intrinsic criterion therefore gives
the stated incidence condition.

Conversely, suppose that $q_\infty$ and $\Pi$ satisfy the incidence
condition.  Let $H_\infty$ be dual to $q_\infty$ and let $P$ be dual to
$\Pi$.  Since $q_\infty\notin\Pi$, one has $P\notin H_\infty$.  For every
$H\neq H_\infty$, choose a point of $\cN\cap\Pi$ on the line joining its
dual point to $q_\infty$.  The corresponding hyperplane $H_P$ contains
$P$, and the same linear-dependence argument gives
$H\cap H_\infty=H_P\cap H_\infty$.  The intrinsic criterion applies.

Finally, if $\cN\cap\Pi$ were contained in a projective line $L\subset\Pi$,
then every point of $\cN$ would belong to the projective plane
$\langle q_\infty,L\rangle$.  This would contradict essentiality.  Hence
$\cN\cap\Pi$ spans $\Pi$.
\end{proof}
For $q_\infty\notin\Pi$, define the \emph{covered set}
\[
 C(q_\infty,\Pi)=
 \left\{q\in\cN\setminus\{q_\infty\}\;\middle|\;
 \langle q_\infty,q\rangle\cap(\cN\cap\Pi)\neq\varnothing\right\}.
\]
Then the special-vertex condition is equivalent to
\[
 \card{C(q_\infty,\Pi)}=\card{\cN}-1.
\]

\begin{corollary}
\label{cor:finite-test}
Let $\cA$ be essential and let $\cN=\cN(\cA)$.  To decide whether $\cA$
is special-vertex, it is enough to enumerate the projective planes spanned
by triples of noncollinear points of $\cN$.
\end{corollary}

\begin{proof}
By the final assertion of Lemma \ref{lem:dual-special-vertex}, a plane occurring
in a special-vertex pair is spanned by $\cN\cap\Pi$.  It therefore contains
three noncollinear points of $\cN$ and appears in the stated enumeration.
\end{proof}

\begin{corollary}\label{cor:special-combinatorial}
For essential arrangements in $\PP^3(\RR)$, the special-vertex property
is determined by the intersection lattice.
\end{corollary}

\begin{proof}
By Lemma~\ref{lem:dual-special-vertex}, if a special-vertex pair exists,
then $\cN\cap\Pi$ spans $\Pi$. Hence
\[
F=\cN\cap\Pi
\]
is a rank-three flat of the matroid represented by $\cN$. The
special-vertex condition is therefore equivalent to the existence of an
element $q_\infty\notin F$ and a rank-three flat $F$ such that
\[
\operatorname{cl}\{q_\infty,q\}\cap F\neq\varnothing
\qquad
\text{for every }q\in\cN\setminus\{q_\infty\}.
\]
This condition depends only on the matroid of the normal configuration,
or equivalently on the intersection lattice of $\cA$.
\end{proof}

We apply the finite test to the projectivized root configurations of types
$D_4$ and $F_4$.  All computations use primitive integer vectors and exact
determinants; in particular, no numerical tolerance or floating-point
recognition enters the argument.

For reproducibility, the enumeration used below is the following.  For a
fixed representative $q_\infty$, we enumerate all unordered triples of
noncollinear points of $\cN(\cA)$, compute the projective plane that they
span, normalize its defining equation to a primitive integer vector, and
deduplicate the resulting planes.  Planes containing $q_\infty$ are
discarded.  For every remaining plane $\Pi$ we compute
$\cN(\cA)\cap\Pi$ exactly, and a point $q\neq q_\infty$ is declared covered
precisely when there is an $r\in\cN(\cA)\cap\Pi$ for which the three normal
vectors representing $q_\infty,q,r$ have rank at most two.  When useful, we group the candidate planes into orbits under the stabilizer of $q_\infty$ in the relevant Weyl group; the numerical entries
below always count the actual candidate planes.  Thus every number in the
tables is obtained from integer rank and determinant computations.

\subsection{The arrangement of type \texorpdfstring{$D_4$}{D4}}

Take the projective normal configuration
\[
 \cN_{D_4}
 =\{[e_i+e_j],[e_i-e_j]\mid 1\leq i<j\leq4\}.
\]
It consists of twelve points.  The Weyl group is transitive on these points,
so it is enough to fix
\[
 q_\infty=[e_1+e_2].
\]
The planes spanned by points of $\cN_{D_4}$ and not containing
$q_\infty$ form five orbits under the stabilizer of $q_\infty$ in
$W(D_4)$, of sizes
\[
1,\qquad 8,\qquad 2,\qquad 2,\qquad 2.
\]
The last three orbits have the same values of
$\card{\cN_{D_4}\cap\Pi}$ and $\card{C(q_\infty,\Pi)}$, so we aggregate
them into a single incidence-data profile. Writing
$[a_1:a_2:a_3:a_4]$ for coordinates in $\PP(V^*)$, the resulting exact
incidence data are as follows.

\begin{center}
\begin{tabular}{c@{\qquad}c@{\qquad}c@{\qquad}c}
\toprule
$\text{representative }\Pi$&\text{number}&
$\card{\cN_{D_4}\cap\Pi}$&$\card{C(q_\infty,\Pi)}$\\
\midrule
$a_1+a_2=0$&1&3&3\\
$a_2-a_4=0$&8&3&6\\
$a_1+a_2-a_3-a_4=0$&6&6&10\\
\bottomrule
\end{tabular}
\end{center}

There are eleven points distinct from $q_\infty$, but the largest covered
set has size ten.  In the last and closest case, one may take the uncovered
point to be $[e_3+e_4]$.  Hence the dual criterion fails.

\begin{proposition}\label{prop:D4}
The Coxeter arrangement $\cA(D_4)=\cA^3_1(12)$ is an irreducible simplicial
arrangement which is not special-vertex.
\end{proposition}

\begin{proof}
Irreducibility and simpliciality are standard for the irreducible finite
Coxeter arrangement of type $D_4$.  The exhaustive plane enumeration above,
together with Corollary \ref{cor:finite-test}, excludes a special-vertex pair.
\end{proof}

\subsection{The arrangement of type \texorpdfstring{$F_4$}{F4}}

A projective root configuration for $F_4$ is
\[
 \begin{split}
 \cN_{F_4}={}&\{[e_i]\mid1\leq i\leq4\}
 \cup\{[e_i\pm e_j]\mid1\leq i<j\leq4\}\\
 &\cup
 \{[(\varepsilon_1,\varepsilon_2,\varepsilon_3,\varepsilon_4)]
 \mid \varepsilon_i\in\{\pm1\}\},
 \end{split}
\]
where opposite sign vectors represent the same projective point.  Thus
$\card{\cN_{F_4}}=24$.

There are two Weyl-group orbits of possible distinguished points: a
long-root orbit represented by $[e_1+e_2]$ and a short-root orbit represented
by $[e_1]$.  For either representative, the exact enumeration of spanned
planes not containing $q_\infty$ gives the same aggregate data:
\[
\begin{array}{c|c|c}
\toprule
\card{\cN_{F_4}\cap\Pi}&\text{number of candidate planes}&
\max\card{C(q_\infty,\Pi)}\\
\midrule
4&80&10\\
9&15&19\\
\bottomrule
\end{array}
\]
The special-vertex criterion would require coverage of all twenty-three
points distinct from $q_\infty$, whereas the exhaustive exact-arithmetic
enumeration above gives a maximum of nineteen.  Hence no candidate plane
from Corollary \ref{cor:finite-test} gives full coverage.

\begin{proposition}\label{prop:F4}
The Coxeter arrangement $\cA(F_4)=\cA^3_1(24)$ is an irreducible simplicial
arrangement which is not special-vertex.
\end{proposition}

\begin{proof}
The two root-length orbits exhaust all choices of $q_\infty$.  For each
orbit, every candidate plane is included by Corollary \ref{cor:finite-test}, and the
coverage data above never attain the required value $23$.  Hence no
special-vertex pair exists.
\end{proof}

\begin{proof}[Proof of Theorem~\ref{thm:main-special-vertex}]
Combine Propositions~\ref{prop:D4} and~\ref{prop:F4} with Examples~\ref{ex:A4} and~\ref{ex:B4}, and Proposition~\ref{prop:B4-chain}.
\end{proof}

\begin{remark}\label{rem:proper}
The failure is not marginal in type $F_4$: for either Weyl-group orbit of
$q_\infty$, every candidate plane $\Pi$ occurring in Corollary
\ref{cor:finite-test} covers at most nineteen of the twenty-three points
distinct from $q_\infty$.  Hence at least four such points remain
uncovered.  Type $D_4$ is closer to the special-vertex condition, since an
optimal candidate covers ten of the required eleven points.
\end{remark}
We conclude this section with a structural question.

\begin{question}
Which irreducible simplicial arrangements in the Gr\"unbaum--Shephard--Geis
catalogue \cite{GrunbaumShephard, Geis} admit a special vertex ?
\end{question}

\section{Purdy-type and Gr\"unbaum--Shephard defects}
\label{sec:defects}
In this section we revisit numerical questions discussed in the final part of
\cite{GrunbaumShephard} and compare them with the special-vertex examples
studied above.  The Purdy-type inequality and the values for the full
Coxeter arrangements $A_4,D_4,B_4,F_4$ were already considered in
\cite{MMPP}.  Here we use those data as a reference point, add the three
proper members of the $B_4$ deletion chain, and compare the resulting
defects with the special-vertex property.  Thus the new contribution of
this section is the interaction with the simplicial deletion chain and the
special-vertex classification, rather than a repetition of the Coxeter
computations from \cite{MMPP}.  Throughout this section, let
\[
\cA=\{H_1,\ldots,H_n\}
\]
be an essential irreducible simplicial arrangement of projective planes
in $\PP^3(\RR)$. We write
\[
\ell(\cA)=|L_2(\cA)|
=
\sum_{q\geq2}t_q^{(1)}
\]
for the number of intersection lines and
\[
p(\cA)=|L_3(\cA)|
=
\sum_{p\geq3}t_p
\]
for the number of intersection points.

We consider two different numerical problems. The first compares the
numbers of rank-two and rank-three flats and leads to a Purdy-type
inequality. The second concerns the distribution of the multiplicities
of the rank-two flats and is related to the Gr\"unbaum--Shephard
conjecture. Although both problems involve intersection lines, they
measure different features of the intersection lattice.

\subsection{Rank-two and rank-three flats}

We first consider the difference
\[
G(\cA)=p(\cA)-\ell(\cA).
\]
Thus the inequality
\[
p(\cA)\geq\ell(\cA)
\]
is equivalent to $G(\cA)\geq0$. This inequality already fails for
irreducible simplicial arrangements. Indeed, for the Coxeter arrangement
of type $A_4$ one has
\[
p\bigl(\cA(A_4)\bigr)=15,
\qquad
\ell\bigl(\cA(A_4)\bigr)=25,
\]
and consequently
\[
G\bigl(\cA(A_4)\bigr)=-10.
\]
The corresponding Purdy-type inequality is weaker. We define the
\emph{Purdy defect} by
\[
\Delta(\cA)
=
p(\cA)-\ell(\cA)+n+2.
\]
Hence
\[
\Delta(\cA)\geq0
\quad\Longleftrightarrow\quad
p(\cA)\geq\ell(\cA)-n-2.
\]

For simplicial arrangements, Theorem~\ref{thm:simpliciality-RP3}
allows us to express $p(\cA)$ directly in terms of the incidence data introduced in
Section~\ref{sec:simplicial-RP3}.

\begin{proposition}
\label{prop:number-of-vertices}
Let $\cA$ be an essential simplicial arrangement of $n$ projective
planes in $\PP^3(\RR)$. Then
\begin{equation}
\label{eq:number-of-vertices}
2p(\cA)
=
\sum_{p\geq3}p\,t_p
-
n
-
\sum_{q\geq3}(q-2)
\sum_{p\geq3}t_{pq}.
\end{equation}
\end{proposition}

\begin{proof}
By Theorem~\ref{thm:simpliciality-RP3},
\[
n
+
\sum_{p\geq3}(2-p)t_p
+
\sum_{q\geq3}(q-2)
\sum_{p\geq3}t_{pq}
=0.
\]
Since
\[
p(\cA)=\sum_{p\geq3}t_p,
\]
we have
\[
\sum_{p\geq3}(2-p)t_p
=
2p(\cA)-\sum_{p\geq3}p\,t_p.
\]
Substituting this into the simpliciality relation and rearranging gives
\eqref{eq:number-of-vertices}.
\end{proof}

Proposition~\ref{prop:number-of-vertices} makes the relation with the
rank-flat and Purdy defects more direct. Since
\[
\ell(\cA)=\sum_{q\geq2}t_q^{(1)},
\]
we obtain
\[
G(\cA)
=
\frac{1}{2}
\left(
\sum_{p\geq3}p\,t_p
-
n
-
\sum_{q\geq3}(q-2)\sum_{p\geq3}t_{pq}
\right)
-
\sum_{q\geq2}t_q^{(1)},
\]
and
\[
\Delta(\cA)
=
\frac{1}{2}
\left(
\sum_{p\geq3}p\,t_p
+
n+4
-
\sum_{q\geq3}(q-2)\sum_{p\geq3}t_{pq}
\right)
-
\sum_{q\geq2}t_q^{(1)}.
\]
In particular, the Purdy inequality $\Delta(\cA)\geq0$ is equivalent
to
\begin{equation}
\label{eq:purdy-GS-data}
\sum_{p\geq3}p\,t_p+n+4
\geq
2\sum_{q\geq2}t_q^{(1)}
+
\sum_{q\geq3}(q-2)\sum_{p\geq3}t_{pq}.
\end{equation}

This formulation also explains why neither the rank-flat inequality nor
the Purdy inequality is an immediate algebraic consequence of the
simpliciality identity alone.  The simpliciality criterion contains the
weighted incidence term
\[
\sum_{q\geq3}(q-2)\sum_{p\geq3}t_{pq},
\]
and therefore the ordinary intersection lines, corresponding to
$q=2$, do not occur in this term.  On the other hand, every ordinary
line contributes to
\[
\ell(\cA)=\sum_{q\geq2}t_q^{(1)}.
\]
Thus the simpliciality identity gives a weighted incidence relation, but
this observation alone does not determine the total number of rank-two
flats.  In particular, it does not rule out the possibility that
simpliciality, together with the remaining incidence relations, implies
the Purdy inequality.

For the arrangements considered in the preceding sections, the
rank-two and rank-three data are collected in
Table~\ref{tab:purdy-defects}.  The rows corresponding to the full Coxeter
arrangements $A_4,D_4,B_4,F_4$ reproduce the values recorded in
\cite{MMPP}; the three intermediate rows are the deletion-chain data used
here.

\begin{table}[ht]
\centering
\renewcommand{\arraystretch}{1.15}
\begin{tabular}{c|r|r|r|r|r|c}
\hline
$\cA$
&
$n$
&
$\ell(\cA)$
&
$p(\cA)$
&
$G(\cA)$
&
$\Delta(\cA)$
&
special vertex
\\
\hline
$\cA(A_4)=\cA^3_1(10)$
&10&25&15&$-10$&2&yes\\
$\cA(D_4)=\cA^3_1(12)$
&12&34&24&$-10$&4&no\\
$\cA^3_1(13)$
&13&40&28&$-12$&3&yes\\
$\cA^3_1(14)$
&14&46&32&$-14$&2&yes\\
$\cA^3_1(15)$
&15&52&36&$-16$&1&yes\\
$\cA(B_4)=\cA^3_1(16)$
&16&58&40&$-18$&0&yes\\
$\cA(F_4)=\cA^3_1(24)$
&24&122&120&$-2$&24&no\\
\hline
\end{tabular}
\caption{Rank-flat and Purdy defects for the arrangements considered
in this paper.}
\label{tab:purdy-defects}
\end{table}

All arrangements in Table~\ref{tab:purdy-defects} have negative
rank-flat difference. Thus none of them satisfies
\[
p(\cA)\geq\ell(\cA).
\]
The values range from
\[
G\bigl(\cA(F_4)\bigr)=-2
\]
to
\[
G\bigl(\cA(B_4)\bigr)=-18.
\]
In contrast, their Purdy defects are all non-negative. In particular,
the additional term $n+2$ changes the picture substantially: the
Coxeter arrangement of type $B_4$, which has the most negative
rank-flat difference among the arrangements in the table, is precisely
the boundary case
\[
\Delta\bigl(\cA(B_4)\bigr)=0.
\]
The simplicial deletion chain
\[
\cA^3_1(13)\subset
\cA^3_1(14)\subset
\cA^3_1(15)\subset
\cA^3_1(16)=\cA(B_4)
\]
exhibits a particularly simple behaviour. For
\[
13\leq n\leq16
\]
one has
\[
p\bigl(\cA^3_1(n)\bigr)=4n-24,
\qquad
\ell\bigl(\cA^3_1(n)\bigr)=6n-38,
\]
and hence
\[
G\bigl(\cA^3_1(n)\bigr)=14-2n,
\qquad
\Delta\bigl(\cA^3_1(n)\bigr)=16-n.
\]
Thus every added plane decreases the rank-flat difference by two and
the Purdy defect by one, with the full $B_4$ arrangement attaining
equality in the Purdy inequality.

The numerical evidence above leads to the following question.

\begin{question}
\label{question:purdy-simplicial}
Does every irreducible simplicial arrangement of projective planes in
$\PP^3(\RR)$ satisfy
\[
p(\cA)-\ell(\cA)+n+2\geq0?
\]
\end{question}

It is also natural to ask about the equality case. Among the
arrangements considered here, the only example satisfying
\[
\Delta(\cA)=0
\]
is the Coxeter arrangement of type $B_4$. Is it the only irreducible simplicial arrangement with $\Delta(\cA)=0$?

\begin{remark}
Micha{\l}ek and the second author constructed in \cite{MMPP} a
hyperplane arrangement
\[
\cA\subset\PP^3(\CC)
\]
for which
\[
\Delta(\cA)<0.
\]
Thus the corresponding inequality does not hold for complex
hyperplane arrangements.
\end{remark}

\subsection{The Gr\"unbaum--Shephard defect}
\label{subsec:GS-defect}

The preceding discussion compares the total numbers of rank-two and
rank-three flats. A different question concerns only the multiplicity
distribution of the rank-two flats.

Recall that $t_q^{(1)}$ denotes the number of intersection lines contained
in exactly $q$ planes. We define the \emph{Gr\"unbaum--Shephard defect}
by
\[
\operatorname{GS}(\cA)
=
t_2^{(1)}-\sum_{q\geq3}t_q^{(1)}.
\]
The Gr\"unbaum--Shephard Conjecture from \cite{GrunbaumShephard} predicts that
\[
\operatorname{GS}(\cA)>0
\]
for every irreducible simplicial arrangement in $\PP^3(\RR)$. Equivalently,
more than one half of the intersection lines should be ordinary, since
\[
\ell(\cA)
=
t_2^{(1)}+\sum_{q\geq3}t_q^{(1)}.
\]
The pair-counting identity
\[
\sum_{q\geq2}\binom{q}{2}t_q^{(1)}
=
\binom{n}{2}
\]
from \eqref{eq:GS-pair-count} gives
\[
t_2^{(1)}
=
\binom{n}{2}
-
\sum_{q\geq3}\binom{q}{2}t_q^{(1)},
\]
and consequently
\begin{equation}
\label{eq:GS-defect}
\operatorname{GS}(\cA)
=
\binom{n}{2}
-
\sum_{q\geq3}
\left(
\binom{q}{2}+1
\right)t_q^{(1)}.
\end{equation}
Hence the Gr\"unbaum--Shephard conjecture is equivalent to the inequality
\[
\binom{n}{2}
>
\sum_{q\geq3}
\left(
\binom{q}{2}+1
\right)t_q^{(1)}.
\]
This formulation provides a useful interpretation of the conjecture.
It says that the contribution of the intersection lines of multiplicity at
least three, when weighted by
\[
\binom{q}{2}+1,
\]
must remain strictly smaller than the total number $\binom{n}{2}$ of
pairs of planes. In particular, intersection lines of high multiplicity
are increasingly restrictive: a $q$-fold line contributes with a weight
which grows quadratically in $q$. Thus the conjecture can be viewed as
asserting that, in an irreducible simplicial arrangement, the
multiplicity distribution of the non-ordinary intersection lines cannot
be too concentrated in high multiplicities.

For completeness, the rank-two multiplicity distributions for the
arrangements considered in this paper are given in
Table~\ref{tab:GS-defects}.  In all seven cases no intersection line has
multiplicity greater than four.

\begin{table}[ht]
\centering
\renewcommand{\arraystretch}{1.12}
\begin{tabular}{c|r|r|r|r}
\toprule
$\cA$ & $t_2^{(1)}$ & $t_3^{(1)}$ & $t_4^{(1)}$ &
$\operatorname{GS}(\cA)$\\
\midrule
$\cA(A_4)=\cA^3_1(10)$&15&10&0&5\\
$\cA(D_4)=\cA^3_1(12)$&18&16&0&2\\
$\cA^3_1(13)$&21&19&0&2\\
$\cA^3_1(14)$&25&20&1&4\\
$\cA^3_1(15)$&30&19&3&8\\
$\cA(B_4)=\cA^3_1(16)$&36&16&6&14\\
$\cA(F_4)=\cA^3_1(24)$&72&32&18&22\\
\bottomrule
\end{tabular}
\caption{Rank-two multiplicities and Gr\"unbaum--Shephard defects for the
arrangements considered in this paper.}
\label{tab:GS-defects}
\end{table}

Thus every arrangement treated here satisfies the strict Grünbaum--Shephard inequality. The conjecture remains open in general, and we are really interested in whether this inequality always holds for irreducible simplicial hyperplane arrangements in three-space.

\section*{Acknowledgments}
We would like to thank Michael Cuntz for introducing us to the notion of hyperplane arrangements with a special vertex. We also would like to thank Paul M\"ucksch for an interesting discussion about supersolvable simplicial arrangements.
\section*{Funding}
Marek Janasz and Piotr Pokora are supported by the National Science Centre (Poland) Sonata Bis Grant  \textbf{2023/50/E/ST1/00025.} For the purpose of Open Access, the authors have applied a CC-BY public copyright license to any Author Accepted Manuscript (AAM) version arising from this submission.

Marek Janasz, Piotr Pokora\\
\noindent
Department of Mathematics,
University of the National Education Commission Krakow,
Podchor\c a\.zych 2,
PL-30-084 Krak\'ow, Poland. \\
Email: \url{marek.janasz@uken.krakow.pl}, \url{piotr.pokora@uken.krakow.pl}.

\begin{thebibliography}{99}

\bibitem{Cuntz}
M. Cuntz, Simplicial arrangements with special vertex, \textbf{arXiv:2607.17785} (2026).

\bibitem{CuntzGeis}
M.~Cuntz and D.~Geis,
\newblock Combinatorial simpliciality of arrangements of hyperplanes,
\newblock \emph{Beitr. Algebra Geom.} \textbf{56(2)}: 439 -- 458 (2015).


\bibitem{DimcaBook}
A. Dimca,
\emph{Hyperplane arrangements: an introduction},
Universitext, Springer, Cham, 2017.

\bibitem{Geis}
D.~Geis, On simplicial arrangements in $\PP^3(\RR)$ with splitting
polynomial, \textbf{arXiv:1902.11185} (2021).

\bibitem{GrunbaumShephard}
B.~Gr\"unbaum and G.~C. Shephard, Simplicial arrangements in projective $3$-space. \emph{Mitt. Math. Sem. Giessen} \textbf{166}: 49 -- 101 (1984).

\bibitem{Hunt}
B. Hunt, Coverings and ball quotients with special emphasis on the 3-dimensional case. \textit{Bonn. Math. Schr.} \textbf{174}: 87 p. (1986).

\bibitem{Melchior}  
E. Melchior, \"{U}ber Vielseite der Projektive Ebene. Deutsche Mathematik {\bf 5}: 461 -- 475 (1941).

\bibitem{MMPP}
M. Micha{\l}ek and P. Pokora, A counterexample to Purdy's inequality for hyperplane arrangements in projective three-space, \textbf{arXiv:2607.08463} (2026).


\bibitem{OrlikTerao}
P.~Orlik and H.~Terao,
\newblock \emph{Arrangements of Hyperplanes},
\newblock Grundlehren der mathematischen Wissenschaften, vol.~300,
Springer-Verlag, Berlin, 1992.

\bibitem{Purdy1981}
G. B. Purdy, A proof of a consequence of Dirac’s conjecture. 
\textit{Geometriae Dedicata} \textbf{10}: 317 -- 321 (1981).

\bibitem{Purdy1986}
G. B. Purdy, Two results about points, lines and planes.
\textit{Discrete Mathematics} \textbf{60}: 215 -- 218 (1986).

\bibitem{KS1}
K. Saito, Theory of logarithmic differential forms and logarithmic vector fields, \textit{J. Fac. Sci., Univ. Tokyo, Sect. I A} \textbf{27}: 265 -- 291 (1980).

\bibitem{terao}
H. Terao, Generalized exponents of a free arrangement of hyperplanes and Shephard--Todd--Brieskorn formula, \textit{Invent. Math.} \textbf{63}: 159 -- 179 (1981).

\bibitem{Yoshinaga}
M. Yoshinaga, Freeness of hyperplane arrangements and related topics,
\textit{Ann. Fac. Sci. Toulouse, Math. (6)} \textbf{23(2)}:  483 -- 512 (2014).

\bibitem{Ziegler}
G.~M. Ziegler,
\newblock Multiarrangements of hyperplanes and their freeness,
\newblock in \emph{Singularities (Iowa City, IA, 1986)},
Contemp. Math. \textbf{90}, Amer. Math. Soc., Providence, RI, 1989,
345--359.

\bibitem{zaslavsky}
Th. Zaslavsky, Facing up to arrangements: Face-count formulas for partitions of space by hyperplanes, \textit{Mem. Am. Math. Soc.} \textbf{154}, 102 p. (1975).

\end{thebibliography}
\end{document}